\documentclass{birkjour}

 \usepackage{amsopn}
 \usepackage{amsmath,amsthm,amssymb}

 \newcommand{\nc}{\newcommand}

\nc{\bb}{\mathfrak{b} }
 \nc{\cc}{\mathfrak{c} }  \nc{\dd}{\mathfrak{d} } 
    \nc{\ggo}{\mathfrak{g} }
 \nc{\hh}{\mathfrak{h} }  \nc{\ii}{\mathfrak{i} }
 \nc{\jj}{\mathfrak{j} }  \nc{\kk}{\mathfrak{k} }
\nc{\mm}{\mathfrak{m} }   \nc{\nn}{\mathfrak{n} }
\nc{\pp}{\mathfrak{p} }   
\nc{\rr}{\mathfrak{r} } \nc{\sg}{\mathfrak{s} }
 \nc{\sso}{\mathfrak{so} }  \nc{\spg}{\mathfrak{sp} }
 \nc{\ssu}{\mathfrak{su} }  \nc{\ssl}{\mathfrak{sl} }
 \nc{\tog}{\mathfrak{t} }  \nc{\uu}{\mathfrak{u} }
 \nc{\vv}{\mathfrak{v} } \nc{\ww}{\mathfrak{w} }
 \nc{\zz}{\mathfrak{z} }  
 
  \newcommand{\ggam}{G/\Gamma}
\nc{\CC}{{\mathbb C}}
 \nc{\DD}{{\mathbb D}}
\nc{\FF}{{\mathbb F}}
\nc{\GG}{{\mathbb G}}  
\nc{\HH}{{\mathbb H}}
\nc{\II}{{\mathbb I}}
\nc{\JJ}{{\mathbb J}}
\nc{\KK}{{\mathbb K}}
\nc{\NN}{{\mathbb N}}

\nc{\RR}{{\mathbb R}}  
 \nc{\ZZ}{{\mathbb Z}}  
 
 \newcommand{\Heis}{\mathrm{H}}

\nc{\ggob}{\overline{\mathfrak{g}}} 
 
\nc{\glg}{\mathfrak{gl} }
  
\nc{\pca}{\mathcal{P}} \nc{\nca}{\mathcal{N}}
 
 \nc{\vp}{\varphi} \nc{\ddt}{\frac{{\rm d}}{{\rm d}t}}
 \nc{\la}{\langle} \nc{\ra}{\rangle}
 \nc{\brg}{[\,,\,]_{\ggo}}
 \nc{\brv}{[\,,\,]_{\vv}}

 \nc{\SO}{{\sf SO}} \nc{\Spe}{{\sf Sp}} \nc{\Sl}{{\sf Sl}}
 \nc{\SU}{{\sf SU}} \nc{\Or}{{\sf O}} \nc{\U}{{\sf U}}
 \nc{\Gl}{{\sf Gl}} \nc{\Se}{{\sf S}} \nc{\Cl}{{\sf Cl}}
 \nc{\Spin}{{\sf Spin}} \nc{\Pin}{{\sf Pin}}

 \nc{\ad}{\operatorname{ad}} \nc{\Ad}{\operatorname{Ad}}
 \nc{\coad}{\operatorname{coad}} 
 \nc{\rank}{\operatorname{rank}} \nc{\Irr}{\operatorname{Irr}}
 \nc{\End}{\operatorname{End}} \nc{\Aut}{\operatorname{Aut}}
 \nc{\Inn}{\operatorname{Inn}} \nc{\Der}{\operatorname{Der}}
 \nc{\Ker}{\operatorname{Ker}} \nc{\Iso}{\operatorname{Iso}}
 \nc{\Le}{\operatorname{L}} \nc{\Fe}{\operatorname{F}}
\nc{\tr}{\operatorname{tr}}
 \nc{\dif}{\operatorname{d}} \nc{\sen}{\operatorname{sen}}
 \nc{\modu}{\operatorname{mod}} \nc{\Ric}{\operatorname{R}}
 \nc{\Sym}{\operatorname{Sym}} \nc{\sca}{\operatorname{sc}}
 \nc{\scalar}{{\sf s}} \nc{\grad}{\operatorname{grad}}
 \nc{\ricci}{\operatorname{r}} \nc{\riccin}{\operatorname{Ric}}
 \nc{\Lie}{\operatorname{L}} \nc{\ct}{\operatorname{T}}

\newcommand{\deax}{\partial_x}
\newcommand{\deay}{\partial_y}
\newcommand{\deaz}{\partial_z}
\newcommand{\deat}{\partial_t}

\nc{\mr}{{\mathfrak r}}
\nc{\ms}{{\mathfrak s}}
\nc{\mv}{{\mathfrak v}}
\nc{\lra}{\longrightarrow}
\nc{\R}{{\mathbb R}}
\nc{\Z}{{\mathbb Z}}

\newcommand{\bsh}{\backslash}

 \theoremstyle{plain}
 \newtheorem{thm}{Theorem}[section]
 \newtheorem{prop}[thm]{Proposition}
 \newtheorem{cor}[thm]{Corollary}
 \newtheorem{lem}[thm]{Lemma}
 
 \theoremstyle{definition}

 \theoremstyle{remark}
 \newtheorem{rem}{Remark}
 
 \newtheorem{exa}[thm]{Example}

 \newcommand{\ri}{{\rm (i)}}
 \newcommand{\rii}{{\rm (ii)}}

\begin{document}
\title[Lorentzian compact manifolds]
{Lorentzian compact manifolds: \\ isometries and geodesics}

\author{V. del Barco}
\email{V. del Barco: delbarc@fceia.unr.edu.ar}
\address{ECEN-FCEIA, Universidad Nacional de Rosario \\Pellegrini 250, 2000 Rosario, Santa Fe, Argentina}

\author{G. P. Ovando}
\email{gabriela@fceia.unr.edu.ar}

\address{CONICET -  Universidad Nacional de Rosario, ECEN-FCEIA, Depto de Matem\'atica, Pellegrini 250, 2000 Rosario, Santa Fe, Argentina}

\author{F. Vittone}
\email{F. Vittone: vittone@fceia.unr.edu.ar}

\address{ECEN-FCEIA, Universidad Nacional de Rosario \\Pellegrini 250, 2000 Rosario, Santa Fe, Argentina}

\date{\today}

\begin{abstract} In this work we investigate families  of 
compact Lorentzian manifolds in dimension four. We show that every lightlike geodesic on such 
spaces is periodic, while there are closed and non-closed spacelike and timelike geodesics. 
Their isometry groups are computed. 
We also show that  there is a non trivial action by isometries of $\Heis_3(\RR)$ on the 
nilmanifold $S^1\times (\Gamma_k \bsh \Heis_3(\RR))$ for $\Gamma_k$ a lattice of $\Heis_3(\RR)$.
\end{abstract}

\thanks{{\it (2010) Mathematics Subject Classification}:  53C50 53C22 57S25
22F30. }

\thanks{ }

\maketitle

\section{Introduction}

Due to their relations with general relativity Lorentzian manifolds, i.e. manifolds endowed with
 metric tensors of index 1, play a special role in pseudo-Riemannian geometry. Timelike and null geodesics represent, respectively,
free falling particles and light rays. Isometric actions and the existence problem of closed or periodic geodesics are two of the most popular topics of research in the last time.

\noindent  The known results developed in the field made use of several techniques including  variational and 
topological methods, Lie theory, etc.. (See for instance \cite{EG,Ge,Ge1,Su,Ti,Ze} and references therein). After the classification of simply connected Lie groups acting locally faithfully  by isometries on 
a compact Lorentz manifold \cite{AS,Ze1} some other questions concerning the geometric implications 
 of such actions  arise in a natural way, specially in the noncompact case (see \cite{PZ}). 
In \cite{Me} Melnick investigated the  isometric actions of Heisenberg groups on compact Lorentzian manifolds, showing  a codimension one action of the Heisenberg Lie group $\Heis_3(\RR)$ on the Lorentzian 
compact solvmanifold $M=\Gamma \bsh G$, where
$G = \RR \ltimes \Heis_3(\RR)$ is a solvable Lie group, called the oscillator group. 

\noindent The main purpose of this work is to analyse these topics more deeply  in a family of examples. We study the geometry of families of compact 
Lorent\-zian manifolds
 in dimension four: $M_{k,i}= G/ \Lambda_{k,i}$, which are stationary, that is, they admit an
everywhere timelike Killing vector field (see \cite{FJP}). We get:
\begin{itemize}
 \item Every lightlike geodesic on any compact space $M_{k,i}$ is periodic, while there are closed and non-closed timelike and spacelike geodesics.
\item The isometry groups of these compact spaces  have  a countable  amount of connected components.
\end{itemize}
As already mentioned  the existence question of closed geodesics on compact Lorent\-zian manifold is 
a classical topic in Lorentzian geometry.
 In this context  the results above relative to null geodesics are
 surprising in a quite different situation of those in \cite{Ga1} and therefore they should induce new research in the topic. 

We start   with  an isometric  codimension one   action by isometries of the Heisenberg Lie group
 $\Heis_3(\RR)$ on  compact nilmanifolds $\Lambda_k \bsh N$ where  $N=\RR \times \Heis_3(\RR)$. 
The starting point is the existence of an isometry between the Lorentzian Lie group $G$ which 
is solvable and the Lie group $N$ which is 2-step nilpotent \cite{BO}. This reveals that
 the existence of actions by isometries coming from non-isomorphic groups does not distinguish
 the isometry class of the Lorentzian manifold.  However while the Lorentzian metric  on $G$ is bi-invariant, that one
 on $N$ is only left-invariant. Furthermore there is a family of groups $\Lambda_k$ which are   cocompact 
lattices of $G$ and also of $N$ so that every quotient $\Lambda_k \bsh N$ is diffeomorphic to $\Lambda_k \bsh G$ and the metrics   induced to the quotients give rise to an isometry 
between 
 the compact space $(\Lambda_k \bsh N, g_N)$ and 
$(\Lambda_k \bsh G, g_G)$.
 It  is clear that as an ideal of $G$, the Heisenberg Lie group $\HH_3$ acts isometrically  on $\Lambda_k \bsh G$
 by translations on the right. Therefore the Heisenberg Lie group also acts on $\Lambda_k \bsh N$ by isometries. 
The Lie group $N$ is already known in the literature: it is related to the known {\em Kodaira-Thurston}  
manifold. 
 One of the advantages of the nilmanifold model arises from the Nomizu's Theorem: the de Rahm 
cohomology can be read off from the cohomology of the Lie algebra of $N$.

\noindent The solvable group $G$ admits more cocompact lattices $\Lambda_{k,i}$ which are not isomorphic to the family above.  We explicitly write the full isometry group of 
$G$ which is proved to be non-compact. And making use of results which relate the isometries on 
the quotients with those on $G$ we compute $\Iso(M_{k,i})$ the group of isometries of the compact 
solvmanifolds $M_{k,i}=\Lambda_{k,i}\bsh G$.
 
 We complete the work with the study  of the periodic geodesics on the  compact Lorentzian  
solvmanifolds. It should be noticed 
that  all the Lorentzian  manifolds here are naturally reductive spaces. We notice that
 together with the motivations coming from Lorentzian geometry an active
 research is given for g.o. spaces  ( see for instance \cite{BCL,CV, Du, Du2}). 
The compact Lorentzian spaces $M_{k,i}$ constitute the first examples (known to us) of compact spaces in 
dimension  four where every 
lightlike geodesic is periodic.


\section{Lorentzian nilmanifolds and actions}

Let $\Heis_3(\RR)$ denote the Heisenberg Lie group of dimension three, which modeled over $\RR^3$ 
has a multiplication map given by
$$(x,y,z) (x',y',z')=(x+x',y+y', z+z'+\frac12(xy'-x'y))$$

Let $N$ denote the nilpotent Lie group  $\RR \times \Heis_3(\RR)$, which turns into a 
 pseudo-Riemannian manifold modeled on $\RR^4$ with the following Lorentz\-ian metric 
\begin{equation}\label{metric}
 g= dt (dz + \frac12 y dx - \frac12 x dy) + dx^2 + dy^2
\end{equation}
where $(t,x,y,z)$ are  usual coordinates  for $\RR^4$. 
Denote  $v=(x, y)$ and for each  $(t_1, v_1, z_1)\in \RR^4$ consider the following 
differentiable function on $\RR^4$:
\begin{equation}\label{actionN}
 L^N_{(t_1, v_1, z_1)}(t_2, v_2, z_2)=(t_1+t_2, v_1+v_2, z_1+z_2+\frac12 v_1^t J v_2)
\end{equation}
where $J$ is the linear map on $\RR^2$ given by the matrix
\begin{equation}\label{mat}
J=\left( \begin{matrix} 0 & 1 \\ -1 & 0 \end{matrix} \right). \qquad \qquad
\end{equation}
Clearly $L^N$ is the translation on the left on $N$ by the element $(t_1, v_1, z_1)$ and it
 is not hard to see that the metric $g$ is invariant under the left-translations 
$L^N_{(t_1, v_1, z_1)}$. A basis of left-invariant vector fields at $p=(t,x,y,z)$ is

\begin{eqnarray*}
e_0(p) & = & \deat|_{p}\\ 
e_1(p) & = & \deax|_{p} -   \frac12 \,y \,\deaz|_{p} \\ 
 e_2(p) & = & \deay|_{p}  + \frac12 \,x  \,\deaz|_{p} \\ 
 e_3(p) & = & \deaz|_{p}
 \end{eqnarray*}
and  the invariant Lorentzian metric $g$ satisfies
$$g(e_0, e_3)=g(e_1, e_1) = g(e_2, e_2)=1.$$

Particular examples of closed subgroups are lattices. A {\em lattice} of a Lie group $G$ is a discrete subgroup $\Gamma$  such
that the quotient space  $G/\Gamma$ or $\Gamma \bsh G$ is compact. 

For every $k\in \NN$ consider  $\Lambda_k$ the following lattice in $N$:
$$\Lambda_{k}  =  2\pi \ZZ \times \Gamma_k <N \quad \mbox{where } \quad \Gamma_k=\ZZ \times \ZZ \times \frac 1 {2k} \ZZ< \Heis_3(\RR)$$
for $\Gamma_k$ a lattice in $\Heis_3(\RR)$.

The metric $g$  on $N$ (\ref{metric}) can be induced to the quotient spaces 
$\Lambda_k \bsh N$. In fact denote also by $g$ the induced metric, 
for every $\gamma \in \Lambda_k$ one has:
$$\begin{array}{rclll}
 g(Z_{\gamma x}, Y_{\gamma x})_{\gamma x} & = & g(d p_{\gamma x}(Z), d p_{\gamma x}({Y}))_{p(\gamma x)}  & \\
&  = & g(d p_x(Z), d p_x(Y))_{p(x)} & =  g(Z_x, Y_x)_x
 \end{array}
$$
  thus the canonical projection $p: N \to \Lambda_k\bsh N$ is a local isometry.

The following proposition shows an action of $\Heis_3(\RR)$ on the compact 
nilmanifolds $\Lambda_k \bsh N$ which is not explained in  \cite{Me}.

\begin{prop}\label{pro1} There is a an isometric action of $\Heis_3(\RR)$ on the compact nilmanifold $\Lambda_k\bsh N$ induced by the action of $\Heis_3(\RR)$ on $N$ given as follows:
\begin{equation}\label{acth}
(v',z') \cdot (t,v,z)= (t, v-R(t)v', z-z'-\frac12 v^{\tau}JR(t)v')
\end{equation}
where  
$R(t)$ is the linear map on $\RR^2$ with  matrix given by
\begin{equation}\label{Rt}
R(t)=\left( \begin{matrix} \cos t & -\sin t \\ \sin t & \cos t \end{matrix} \right) \qquad t\in \RR.
\end{equation}
\end{prop}

The proof follows from several computations which can be done by hand:  
for every $(v',t')\in \Heis_3(\RR)$ the map above (\ref{acth}) defines 
an isometry on $N$ which can be induced to $\Lambda_k \bsh N$. This gives rise 
to an action of $\Heis_3(\RR)$ on the nilmanifold $\Lambda_k\bsh N$.
In next sections we shall explain the construction of the action above (see Remark \ref{6}). 
 
\begin{rem}
The action of $\Heis_3(\RR)$ by isometries on the quotient $\Lambda_k\bsh N$ is neither  
induced by the translations on the left nor on the right on $N$. 
\end{rem}

The orbits of the action of $\Heis_3(\RR)$ on $N$ are parametrized by $t_0\in \RR$:
$$\mathcal O_{(t_0, v_0, z_0)}=\{(t_0, v, z)\in \RR^4\, v\in \RR^2, \, z\in \RR\}$$
and they are not totally geodesic  except for $t=0$ (see geodesics in the next section). 

On $\RR^4$ consider the lightlike distribution
$$\mathcal D_{p}= span\{e_1, e_2, e_3\},$$
which is involutive. Integral submanifolds for $\mathcal D$ are given by the orbits $\mathcal O_p$.

\section{A Lorentzian solvable Lie group}



Recall that if $G$ is a connected real Lie group, its Lie algebra $\ggo$ is identified with the Lie algebra of left-invariant vector fields on $G$.  Assume $G$ is 
  endowed with a left-invariant pseudo-Riemannian metric $\la\,,\,\ra$. Then 
  the following statements are equivalent (see \cite[Ch. 11]{ON}):

\begin{enumerate}
\item $\la\, ,\,\ra$  is right-invariant, hence bi-invariant;
\item $\la\,,\,\ra$  is $\Ad(G)$-invariant;
\item  the inversion map $g \to g^{-1}$ is an isometry of $G$;
\item $\la [X, Y], Z\ra + \la Y, [X, Z] \ra= 0$ for all $X,Y,Z \in \ggo$;
\item $\nabla_XY = \frac12 [X, Y]$ for all $X,Y  \in \ggo$, where
$\nabla$ denotes the Levi Civita connection;
\item  the geodesics of $G$ starting at the identity element $e$ are the 
one parameter subgroups of $G$.
\end{enumerate}

 By (3) the pair $(G, \la\,,\,\ra)$ is a pseudo-Riemannian symmetric space. Furthermore by computing
   the curvature tensor one has
 \begin{equation}
 R(X, Y) = - \frac14 \ad([X, Y]) \qquad \quad  \mbox{ for } X,Y \in
 \ggo.
 \label{curvatura}
 \end{equation}
Thus the Ricci tensor $Ric(X,Y)=tr( Z \to R(Z,X)Y )$ is given by
$$Ric(X,Y)=-\frac14 B(X,Y)$$
where $B$ denotes the Killing form on $\ggo$ given by $B(X,Y)=\tr (\ad(X) \circ \ad(Y))$ for all 
$X,Y \in \ggo$, $\tr$ denotes the usual trace.

Consider the Lie group homomorphism $\rho: \RR \to \Aut(\Heis_3(\RR))$ which on
vectors $(v, z)\in \RR^2 \oplus \RR$ has the form
\begin{equation}\label{R0}
\rho(t)=\left( \begin{matrix}
R(t) & 0 \\
0 & 1 \end{matrix} 
\right) \qquad   \mbox{ where } \qquad 
R(t)  =  {\left( \begin{matrix}
\cos \,t & -\sin \,t \\ \sin \,t & \cos \,t \end{matrix} \right) }.
\end{equation}

Let  $G$  denote the simply connected Lie group which is modelled on the smooth
 manifold $\RR^4$, where the algebraic structure is the resulting from the
semidirect product of $\RR$ and $\Heis_3(\RR)$, via $\rho$. Thus the multiplication is given by
\begin{equation}\label{oper}
 (t,v,z) \cdot (t',v',z')  = (t+t', v+ R(t)v',
z+z'+\frac12 v^T JR(t) v') 
\end{equation}
with $J$ and $R(t)$ as above. The Lie group $G$ is known as the {\em oscillator group}. 

A basis of left-invariant vector fields  at a point $p=(t,x,y,z)$ is given by

\begin{eqnarray*}
X_0(p) & = & \deat|_{p}\\ 
X_1(p) & = & \cos\, t  \,\deax|_{p} + \sin \,t  \,\deay|_{p}
 + \frac12(x \,\sin\, t - y \,\cos \,t) \,\deaz|_{p} \\ 
 X_2(p) & = & -\sin \,t \, \deax|_{p} + \cos \,t \,\deay|_{p}
 + \frac12(x \, \cos\, t + y\, \sin \,t) \,\deaz|_{p} \\ 
 X_3(p) & = & \deaz|_{p} .
 \end{eqnarray*}

These vector fields verify the Lie bracket relations:
 \begin{equation}
[X_0 , X_1 ] = X_2 \quad [X_0 , X_2 ] = -X_1 \quad [X_1 , X_2 ] = X_3
\label{lbg0}
\end{equation}
giving rise to the Lie algebra of $G$, namely $\ggo$. 
On the usual basis of $T_p G$, $\left\{ \deat|_p,\;\deax|_p,\;\deay|_p,\;\deaz|_p\right\}$   the matrix:
\begin{equation}
\left(\begin{matrix}
0 & \frac{1}{2}y &-\frac{1}{2}x& 1\\
\frac{1}{2}y & 1 & 0 & 0\\
-\frac{1}{2}x & 0 & 1 & 0\\
1 & 0 & 0 & 0
\end{matrix}\right);\qquad
\label{gmatrixG0}
\end{equation}
defines a bi-invariant metric on $G$. On canonical coordinates of $\RR^4$ it corresponds to the  
 pseudo-Riemannian metric:
$$
g=dz\,dt+dx^2+dy^2+\frac12(y dx\,dt-x dy\,dt),
$$
which coincides with the metric $g$ (\ref{metric}). 
 
\begin{prop} The Lorentzian manifold $(\RR^4, g)$ for $g$ the Lorentzian metric in (\ref{metric}) 
admits simply and transitive actions of both Lie groups $N$ and $G$.

As a consequence $(N, g)$ is isometric to $(G, g)$.
\end{prop}

In fact one can see that starting at $(0,0,0,0)\in \RR^4$ the translation on the left (by $N$ or $G$) 
gives the same Lorentzian metric at every point. See \cite{BO}.

\begin{rem} While the metric $g$ is left and right-invariant on $G$, the metric $g$ is only 
left-invariant on $N$. In particular $(G, g)$ and $(N, g)$ is a symmetric space: geodesics are one-parameter groups.
\end{rem}

\begin{rem}
 The Lie group $G$ is the isometry group of a left-invariant Lorentzian metric on the Heisenberg Lie group $\Heis_3(\RR)$ (see \cite{BOV}).
\end{rem}

\subsection{Isometries}  Let $G$ be a connected Lie group with a bi-invariant metric, and let 
    $\Iso(G)$ denote the
isometry group of $G$. This is a Lie group when endowed with the compact-open topology. 
Let $\varphi$ be an isometry such that $\varphi(e)=x$, for $x\neq e$. Then
 $L_{x^{-1}} \circ \varphi$ is an isometry which fixes the element $e\in G$. 
Therefore $\varphi=L_{x} \circ f$ where $f$ is an isometry such that $f(e)=e$.  Let $\Fe(G)$ denote 
the isotropy subgroup of the identity $e$ of $G$ and let
$\Le(G) := \{L_g : g \in G\}$, where $L_g$ is the  translation on the left by $g\in G$. 
Then $\Fe(G)$ is a closed subgroup of $\Iso(G)$ and 
\begin{equation}\label{desciso}
\Iso(G) = \Le(G) \Fe(G) = \{L_g \circ f : f \in \Fe(G), g\in G\}.\end{equation}
Thus $\Iso(G)$ is essentially determined by $\Fe(G)$.

The bi-invariant metric on $G$ implies that it is  a  symmetric space. For locally symmetric spaces one has the 
 Ambrose-Hicks-Cartan theorem (see for example \cite[Thm. 17, Ch. 8]{ON}), which states that
 on a complete locally symmetric 
pseudo-Riemannian   manifold $M$, a linear isomorphism
$A : T_p M \to T_p M$ is the differential of some isometry of $M$ that fixes the point $p$ if and 
only if it preserves the scalar product that the metric induces into the tangent space and if for 
every $u,v, w \in T_p M$
the following equation holds:
$$ R(Au, Av)Aw = AR(u, v)w.$$

By applying this to the Lie group $G$ equipped with a bi-invariant metric and 
whose  curvature formula was given in  (\ref{curvatura}) one gets the next result (see also \cite{Mu}).

 \begin{lem} \label{iso} Let $G$ be a simply connected Lie group with a bi-invariant 
pseudo-Riemannian metric $\la\,,\,\ra$. Then a linear isomorphism $A : \ggo \to \ggo$  
is the differential of some isometry in $\Fe(G)$ if and only if for all $X,Y, Z\in \ggo$, 
the linear map $A$ satisfies the following two conditions:

\vskip 3pt

\ri \quad $\la A X, A Y \ra = \la X, Y\ra $;

\vskip 3pt

\rii \quad $ A[[X, Y], Z] = [[AX, AY], AZ]$.

\vskip 3pt
\end{lem}

Whenever $G$ is simply connected, every local isometry of $G$ extends to a unique global one. Therefore the full group of isometries of $G$ fixing the identity is isomorphic 
to the group of linear isometries of $\ggo$ that satisfy the conditions  of Lemma \ref{iso}. 
By applying this to our case, one gets the next result (see \cite{BOV}).

\begin{thm} \label{tiso} Let $G$ be the simply connected solvable Lie group 
of dimension four $\RR\ltimes_{\rho} \Heis_3(\RR)$ endowed with the bi-invariant metric $g$. 
Then the group of
isometries fixing the identity element $\Fe(G)$ is isomorphic to
     $(\{1,-1\}\times \Or(2))\ltimes \RR^2$.

In particular the connected component of the identity of $\Fe(G)$ coincides 
with the group of inner automorphisms $\{\chi_g:G\to G,\; \chi_g(x)=gxg^{-1}\}_{g\in G}$.
\end{thm}
 
The computations (see  \cite{BOV}) show that the differential of an isometry fixing the identity element corresponds to
 $A: \ggo \to \ggo$
 having the following 
matricial presentation  on the basis of left-invariant vector fields $\{X_0, X_1, X_2 , X_3\}$ 
 \begin{equation}\label{isom}
 A=\left( 
 \begin{matrix}
 \pm 1 & 0 & 0 \\
 w & \tilde{A} & 0 \\
 \mp\frac12 ||w||^2 & \mp w^{\tau} \tilde{A} & \pm 1
 \end{matrix}
 \right)
 \end{equation}
where $w\in \RR^2$ and $\tilde{A}\in \Or(2)$.  This gives a group isomorphic to 
$(\{1,-1\}\times \Or(2))\ltimes \RR^2$ for which the identity component corresponds to 
those matrices of the form (\ref{isom}) with
 $a_{00}=a_{33}=1$ and $\widetilde{A}\in \SO(2)=$$\{R(t):t\in\RR\}$.
 
  On the other hand, the set of orthogonal automorphisms of $\ggo$ coincide with the set $\Ad(G)$,
 that is, the matrices of the form  
  $$\Ad(t,v)= \left( \begin{matrix}
  1 & 0 & 0 \\
  Jv & R(t) & 0 \\
  -\frac12 ||v||^2 & -(Jv)^{\tau} R(t) & 1
  \end{matrix}
  \right), \qquad  v\in \RR^2.
  $$
being $A(t,v)=\Ad(t,v,z)$ for $v=(x,y)$. Since both subgroups are connected and have the same dimension, they must coincide.

 \begin{rem} In \cite{BO} more features about the isometry group of $(G,g)$ were studied. It was proved that
 $N = \RR \times \Heis(\RR)$ occurs as a subgroup of $\Iso(G)$ but it is not contained into the
 nilradical of $\Iso(G)$. Furthermore the action of the nilradical on $G$ is not transitive. This shows important differences between the  Riemannian situation and the Lorentzian case, even for 2-step nilpotent Lie groups. 
 \end{rem}
 
 Now we proceed to write explicitly the isometries on $G$.
Since $\Fe(G)$ has four connected components, our aim is to find a representative isometry on
 each of them. 

From Theorem \ref{tiso}, the connected component of the identity 
$$\Fe_0(G)=\{\chi_g:\;g\in G\}\simeq\left(\{1\}\times \SO(2)\right)\ltimes \RR^{2};$$ 
where if $g=(t_0,v_0,z_0)$, with $v_0=(x_0,y_0)$, then for $v=(x,y)$ 
\begin{equation}\label{ig}
\begin{array}{rcl}
\chi_g(t,v,z) & = & (t,v_0+R(t_0)v-R(t)v_0,\\
& & z+\frac12 v_0^{\tau}JR(t_0)v-\frac12 v_0^{\tau}JR(t)v_0-\frac12(R(t_0)v)^{\tau}JR(t)v_0).
\end{array}
\end{equation}

Consider the  semidirect product $G \ltimes G$ given by conjugation: $g \cdot h= \chi_g(h)$ as above.
Then $G \ltimes G$ acts by isometries on the pseudo-Riemannian manifold
 $G$, the first factor acts by conjugation $\chi: G \to \Fe_0(G)$ and the second one by translations
 on the left $L: G \to \Le(G)$, however  this action   is not effective. Since 
$$\chi_g \circ L_h \circ \chi_{g^{-1}}=L_{\chi_g(h)}\qquad (*)$$ the action induces the group homomorphism:
$$G \ltimes G \to \Iso(G) \qquad (h, g) \mapsto L_g \circ \chi_h.$$

The  homomorphism  $\chi: G \to \Fe_0(G)$ has the center of $G$ as kernel
$$Z(G)=\{g\in G \,:\, g x g^{-1}= x \quad \mbox{ for all } x \in G\}$$
and one  gets 
\begin{equation}\label{ker}\Fe_0(G)\simeq G/Z(G) \simeq \SO(2) \ltimes \RR^{2}.
 \end{equation}
 It is not hard to see that the center of $G$ is the subgroup  generated by the element 
of $(0,0,0, 1)$. On the other hand 
 the subgroup $\Le(G)$ is normal in $\Iso_0(G)$ and  the group homomorphism $L: G \to \Le(G)$ has trivial kernel. 
 
Thus  the connected component of the identity (isometry) is 
$\Iso_0(G)=(\SO(2)\ltimes \RR^2) \ltimes G.$

Let $f_1,f_2,f_3:G\to G$ denote the following diffeomorphisms:
\begin{eqnarray}
f_1(t,v,z)&=&(-t,Sv,-z),\ \ \text{ where } S(x,y)=(-x,y)\label{f1}\\ 
f_2(t,v,z)&=&(-t,R(t)v,-z), 
\label{f2}\\ 
f_3(t,v,z)&=&f_1\circ f_2(t,v,z)=(t,R(t)Sv,z) 
\label{f3}
\end{eqnarray}

Usual computations show that $f_i$ is an isometry for $i=1,2,3$ and they belong to different connected components of the
 isometry group.  Thus the other three components of $\Fe(G)$ are 
 $$ \Fe_0(G)\cdot f_1,\qquad \Fe_0(G)\cdot f_2 \qquad\text{and }\ \Fe_0(G)\cdot f_3$$
where $F_0 \cdot f_i=\{ g f_i \, : \, g\in F_0(G)\}$.

\subsection{Geodesics}  \label{geodes} From (\ref{gmatrixG0}) one can  compute the Christoffel 
symbols of the Levi-Civita connection (cf. \cite{ON}) and therefore a curve
 $\alpha(s)=(t(s),x(s),y(s),z(s))$ is a geodesic on $G$ if its components satisfy the 
 second order system of differential equations:

$$\left\{\begin{array}{rcl}
t''(s)&=&0,\\
x''(s)&=&-t'(s)y'(s),\\
y''(s)&=&t'(s)x'(s),\\
z''(s)&=&\frac{1}{2}\;t'(s)(x(s)x'(s)+y(s)y'(s)).
\end{array}\right.
$$ 

On the other hand, if $X_e=\sum_{i=0}^{3}a_{i}X_{i}(e)\in T_{e} G$, then the geodesic
 $\alpha$ through $e$ with initial condition $\alpha'(0)=X_e$ is the integral curve of the 
left-invariant vector field $X=\sum_{i=0}^{3}a_{i}X_{i}$. Then we should have $\alpha'(s)=X_{\alpha(s)}$. 

$\bullet$ If $a_{0}\neq 0$ the components of $\alpha$ must verify the following system
\begin{eqnarray*}
t'(s)&=&a_{0},\\
x'(s)&=&a_{1}\cos a_{0}s-a_{2}\sin a_{0}s,\\
y'(s)&=&a_{1}\sin a_{0}s+a_{2}\cos a_{0}s,\\
z'(s)&=&\frac{1}{2}\left[\frac{a_{1}^{2}}{a_{0}}+\frac{a_{2}^{2}}{a_{0}}+2a_{3}-\left(\frac{a_{2}^{2}}{a_{0}}+\frac{a_{1}^{2}}{a_{0}}\right)\cos a_{0}s\right].
\end{eqnarray*}
and so the geodesic through $e=(0,0,0,0)$ with initial condition $X_e$ satisfies:
\begin{eqnarray*} \label{geodcomp}
t(s)&=&a_{0}s,\\
x(s)&=&\frac{a_{1}}{a_{0}}\sin a_{0}s+\frac{a_{2}}{a_{0}}\cos a_{0}s-\frac{a_{2}}{a_{0}},\\
y(s)&=&-\frac{a_{1}}{a_{0}}\cos a_{0}s+\frac{a_{2}}{a_{0}}\sin a_{0}s+\frac{a_{1}}{a_{0}},\\
z(s)&= &\frac{1}{2}\left[ \left(\frac{a_{1}^{2}}{a_{0}}+\frac{a_{2}^{2}}{a_{0}}+2a_{3}\right)s-\left(\frac{a_{2}^{2}}{a_{0}^{2}}+\frac{a_{1}^{2}}{a_{0}^{2}}\right)\sin a_{0}s \right].\end{eqnarray*}

 If $a_{0}=0$, it is easy to see that $\alpha(s)=(0,a_{1}s,a_{2}s,a_{3}s)$ is the corresponding geodesic.
 
Therefore the exponential map $\exp: \ggo \to G$ is 
$$\exp (X)=
  \displaystyle{\left(a_0,\frac{1}{a_0}(R_0(a_0)J-J) (a_1,a_2)^{\tau},a_3 +\frac{1}{2}\left(\frac{a_1^2}{a_0}+\frac{a_2^2}{a_0}\right)\left(1-\frac{\sin a_0}{a_0}\right)\right)}$$
for $a_{0}\neq 0$, while if $a_{0}= 0$, 
 $$\exp (X)= \displaystyle{ \left(0,a_1,a_2,a_3\right)}.$$

 The geodesic passing through the point $h\in G$,  is the translation on the left by $h$ of the
 one-parameter group at $e$, that is $\gamma(s)= h \exp(sX)$ for $\exp(sX)$ given above. 
  
    

\section{Lorentzian compact manifolds}

Let  $K$ denote a  closed subgroup of $G$ so that  $G/K$ is a differentiable manifold  
endowed with a $G$-invariant metric, that is, a metric such that  the transformations $\tau_h: G/K  \to G/K$ given by $\tau_h (xK)=hx K$ are isometries for all $h\in G$ and such that the natural projection $p:G\to G/K$ is a pseudo-Riemannian submersion. Thus  
 $$\widetilde{\Le}(G/K)= \{\tau_h : h\in G\}$$
   is a subgroup of the isometry group $\Iso(G/K)$ of the quotient space.

If $f\in \Iso(G)$ is an isometry of $G$ we say that $f$ is \textsl{fiber preserving}  if 
$f(gK)=f(g)K$ for every $g\in G$. 
If $f$ is a fiber preserving isometry of $G$, it induces an isometry $\widetilde{f}$ of $G/K$ 
defined by $\widetilde{f}(gH)=f(p(g))$. Observe that left-translations in $G$ are 
fiber preserving and they induce the isometries $\tau_h$ in $G/K$. 

\begin{exa} \label{exis} Let $\Gamma<G$ be a lattice of a Lie group $(G, g)$ which is equipped
 with a bi-invariant metric. Then the metric $g$ of $G$ is induced to both quotients $(G/\Gamma, g)$ and $(\Gamma\bsh G, g)$ (by abuse we name the induced metrics also by $g$). Since the inversion map: $G\to G$ which sends $h \to h^{-1}$ is an isometry of $G$, one induces this map to the quotients: $x\Gamma \to \Gamma x^{-1}$ and one gets that $G/\Gamma$ and $\Gamma \bsh G$ are isometric compact spaces. This isometry enables the computation of the geometry without distinguishing these spaces. Furthermore $G$ acts by isometries on $G/\Gamma$ on the left via the maps $\tau_h$ (as before), $G$ acts isometrically on $\Gamma\bsh G$ on the right $h \cdot \Gamma x= \Gamma x h^{-1}$. 
\end{exa}

\begin{lem} Let $G$ be a Lie group with a bi-invariant metric and let 
$\Gamma$ be a lattice of $G$. Then $G/\Gamma$ admits a $G$-invariant 
metric making of it a naturally reductive pseudo-Riemannian space  and consequently:

\begin{enumerate}
\item $p: G \to G/\Gamma$ is a pseudo-Riemannian covering;

\item The geodesics in $G/\Gamma$  starting at  the point $o = p(e)$ are of the form 
$p(\exp \,tX)$ with $X\in \ggo$. 
\end{enumerate}
\end{lem}
See \cite[Ch. X vol. 2]{KN}, \cite{ON}.

We can study  the isometry group of $G/\Gamma$ once one has information about the
 isometry group of $G$, $\Iso(G)$ as follows.

\begin{thm}
Let $G$ be an arcwise-connected, simply connected Lie group with a left-invariant metric
 and $\Gamma$ a discrete subgroup of $G$. Then every isometry ${f}$ of $G/\Gamma$ is induced 
to $G/\Gamma$ by a fiber preserving isometry of $G$. 
\label{fiberpreserving}
\end{thm}
\begin{proof}
Let $f\in \Iso(G/\Gamma)$ and consider $f\circ p:G\to G/\Gamma$. Since $G$ is simply connected, from the Lifting Theorem (cf. \cite[Ch. III, Th. 4.1]{Bre}), there exists a differentiable map $\phi:G\to G$ such that 
\begin{equation}
p\circ \phi=f\circ p.
\label{ind}
\end{equation} 
From the construction of $\phi$ it is not difficult to see that $\phi$ is a diffeomorphism of $G$ 
if $f$ is a diffeomorphism of $G/\Gamma$. Since the projection $p:G\to G/\Gamma$ is a 
pseudo-Riemannian covering map one gets that $\phi$ is a local isometry and therefore an 
isometry. From (\ref{ind}) it is immediate that $\phi$ is fiber preserving and $f$ is induced by $\phi$. 
\end{proof}

Recall that the Lie algebra of the Isometry group is obtained from the Killing vector fields. 
The next lemma states a relationship between the Killing vector fields on $G$ and those on $G/\Gamma$,
 for a lattice $\Gamma < G$.

\begin{lem}\label{kil}
Let $G$ be a Lie group with a left-invariant metric and $\Gamma$ a discrete closed subgroup of $G$. Let $X$ be a 
Killing vector field in $\ggam$ with monoparametric subgroup $\{\varPsi_{t}\}$. Then the horizontal lift 
$\overline{X}$ to $G$ of $X$ (with respect to the pseudo-Riemannian submersion $p:G\to \ggam$)
 is a Killing vector field on $G$ whose monoparametric subgroup $\{\varphi_{t}\}$ verifies
$$\varPsi_{t}\circ p= p \circ\varphi_{t}$$
\end{lem}
\begin{proof}
Let $\mathfrak{iso}(\ggam)$ and $\mathfrak{iso}(G)$ denote the Lie algebras of the isometry groups of 
$\ggam$ and $G$ respectively. Since $G$ and $\ggam$ are complete, the Lie algebras $\mathfrak{iso}(\ggam)$ and $\mathfrak{iso}(G)$ can be identified with the
 corresponding Lie algebras of Killing vector fields. Therefore, if $\varPsi$ belongs to 
$\Iso_0(\ggam)$ there exist Killing fields $X_{1},\cdots, X_{n}$ in $\ggam$ with
 monoparametric subgroups $\{\varPsi_{t}^{i}\}$ such that 
$$ \varPsi= \varPsi^{1}_{1}\circ\cdots\circ \varPsi^{n}_{1}.$$ 
Let $\overline{X_{i}}$ be the horizontal lift to $G$ of $X_{i}$ (with respect to the 
pseudo-Riemannian submersion
 $p:G\to\ggam$), $i=1,\cdots,n$,  and let $\{\varphi_{t}^{i}\}$ be the associated
 monoparametric subgroups. Let
 $f=\varphi_{1}^{1}\circ\cdots\circ\varphi_{1}^{n}\in \Iso_0(G)$. 

Fix $q\in \ggam$ and let $\sigma_{n}$ be a local section of $p:G\to\ggam$ defined on a neighborhood of $q$ and for each $i=1,\cdots,n-1$, let $\sigma_{i}$ be a local section around $q_{i}=\varPsi_{1}^{i+1}\circ\cdots\circ\varPsi^{n}_{1}(q)$, mapping $q_{i}$ into $\varphi^{i+1}_{1}\circ\cdots\circ\varphi^{n}_{1}(\sigma_{n}(q))$. Then, we must have
$$\varPsi=p \varphi^{1}_{1}\sigma_{1}\circ\cdots\circ p \varphi^{n}_{1}\sigma_{n}=p\circ f\circ 
\sigma_{n}.$$
This decomposition in independent of the choice of the local section and in fact, $$\varPsi\circ p=p\circ f.$$
\end{proof}

\begin{rem}
By the previous lemma any isometry in $\Iso_{0}(\ggam)$ is induced to the quotient by 
an isometry in $\Iso_{0}(G)$.
\end{rem} 

We concentrate our attention now to  the solvable Lie group $G$ equipped with the bi-invariant metric $g$ given in 
(\ref{gmatrixG0}). We shall construct compact manifolds and study their geometry. Consider the following  lattices of $G$. 


Set $\Gamma_k$ the lattice of the Heisenberg Lie group $\Heis_3(\RR)$ given by
$$\Gamma_k=\ZZ \times \ZZ \times \frac 1 {2k} \ZZ\, \qquad k\in \NN.$$

Every lattice $\Gamma_k$ is invariant under the subgroups generated by $\rho (2\pi)$, $\rho (\pi)$ 
and $\rho (\frac \pi 2)$, ($\rho:\RR \to Aut(\Heis_3(\RR)$ as in (\ref{R0})). Consequently we
 have three families of lattices in $G=\RR  \ltimes_{\rho} \Heis_3 (\RR)$:

\begin{eqnarray} 
\Lambda_{k,0} & = & 2\pi \ZZ \ltimes \Gamma_k < G\nonumber\\
\Lambda_{k,\pi} & = & \pi \ZZ \ltimes \Gamma_k < G\label{lati}\\
\Lambda_{k,\pi/2} & = & \frac \pi 2 \ZZ \ltimes \Gamma_k < G.\nonumber
\end{eqnarray}

so that $\Lambda_{k,0} \triangleleft \Lambda_{k,\pi} \triangleleft \Lambda_{k,\pi/2}$, which induce the solvmanifolds
\begin{equation}\label{solvm}
\begin{array}{rcl}
M_{k,0} & =& \Lambda_{k,0}  \bsh G \simeq G/\Lambda_{k,0}\, ,\\
M_{k,\pi} & = & \Lambda_{k,\pi} \bsh G\simeq G / \Lambda_{k, \pi}\, ,\\
M_{k,\pi/2} & = & \Lambda_{k,\pi/2} \bsh G\,\simeq G/\Lambda_{k, \pi/2} .
  \end{array}
\end{equation}

Since subgroups $\Lambda_{k,i}$ are not pairwise isomorphic (see for instance \cite{COS}), 
 they determine non-diffeomorphic solvmanifolds (see for instance \cite{Ra}).

Observe that the action of $\rho(2\pi)$ is trivial, so 

\begin{itemize}
\item $\Lambda_{k,0}=2\pi \ZZ\times \Gamma_k$ (a direct product) and 
\item $M_{k,0}=G/\Lambda_{k,0}$ is diffeomorphic to $\Lambda_{k,0}\bsh G \simeq \Lambda_k\bsh N \simeq S^1 \times \Heis_3(\RR)/\Gamma_k$, a {\em Kodaira Thurston manifold} (see more details in \cite{COS}).
\end{itemize}

Moreover every compact space in the family $M_{k,0}$ admits a symplectic but non-K\"ahler structure,
 but any compact space $M_{k,i}$ $i=\pi, \pi/2$ admits no symplectic structure since the second Betty number
 vanishes (see \cite{COS}). 

\begin{prop} The compact solvmanifolds $M_{k,i}$ for $k\in \NN$ and $i=0, \pi, \pi/2$ are
pseudo-Riemannian  naturally reductive spaces, hence complete.

The solvable Lie group $G=\RR \ltimes \Heis_3(\RR)$ acts by isometries on each of the 
compact spaces $M_{k,i}$ for $k\in \NN$ and $i=2\pi, \pi, \pi/2$.
As a consequence the Heisenberg Lie group $\Heis_3(\RR) < G$ also acts on each of the
 compact spaces $M_{k,i}$ for $k\in \NN$ and $i=0, \pi, \pi/2$.

Both actions are locally faithful.
\end{prop}

\begin{rem} \label{6} The action of $\Heis_3(\RR)$ on $\Lambda_k \bsh N$ of Proposition \ref{pro1} is induced by the
 right action of $G$ on $M_{k,0}\simeq \Lambda_{k,0}\bsh G \simeq \Lambda_k \bsh N$:
$$(v', z')\cdot \Lambda_{k,0}(t,v,z)=\Lambda_{k,0}((t,v,z) (0,v',z')^{-1})$$
where on the right side we are considering the multiplication map of $G$. Since the metric 
is bi-invariant the right-translation is also an isometry. 
\end{rem}

\subsection{Isometries of the compact spaces $M_{k,s}$}
Our goal now is to study  the isometry groups of the compact spaces $M_{k,s}$. 

 Notice that all translations on the left $L_h$ for $h\in G$
 are fiber preserving isometries. Direct computations show that the only isometries in $\Fe(G)$ 
that are fiber preserving are the inner homomorphisms $\chi_{h}$ with 
$h\in \mathcal N_{G}(\Lambda_{k,s})$, the normalizer of $\Lambda_{k,s}$ in $G$.

\begin{lem} \label{normal} For the lattices of $G$ described in (\ref{lati}), namely
 $\Lambda_{k,i}$ for every $k\in \NN$, let $M_{k,i}=G/\Lambda_{k,i}$.

\begin{itemize}

\item The only isometries in $F(G)$ that are fiber preserving are the inner
 homomorphisms $\chi_{h}$ with $h\in \mathcal N_{G}(\Lambda_{k,s})$. 

\item The normalizers in $G$ of these lattices are given by
\begin{enumerate}
\item $\mathcal N_{G}(\Lambda_{k,0})= \frac{\pi}{2}\ZZ\ltimes (\frac{1}{2k}\ZZ \times \frac{1}{2k}\ZZ\times \RR),$
\item $\mathcal N_{G}(\Lambda_{k,\pi})= \frac{\pi}{2}\ZZ\ltimes (\frac{1}{2}\ZZ \times \frac{1}{2}\ZZ\times \RR)$,
\item Set $\mathcal W=\{(m,n)\in\ZZ^2: m \equiv n \;(\mbox{mod } 2)\}$ then
$$\mathcal N_{G}(\Lambda_{k,\frac{\pi}{2}})= \left\{\begin{array}{ll}
\frac{\pi}{2}\ZZ\ltimes (\mathcal W\times \RR) & \mbox{ for }k=1,\\
\frac{\pi}{2}\ZZ\ltimes (\frac{1}{2}\mathcal W\times \RR) & \mbox{ for } k\geq 2.\end{array}\right.
$$

\end{enumerate}
\end{itemize}
\end{lem}

\begin{proof}
Let $\Lambda_{k,0}$ be the lattice of $G$ given in (\ref{lati}). Let $g=(t_0,v_0,z_0)\in G$ with $v_0=(x_0,y_0)\in\R^2$, be an element in the normalizer of $\Lambda_{k,0}$. 
Let $\gamma=(t,v,z) \in \Lambda_{k,0}$ where $v=(x,y)$. Thus from the formulas in (\ref{ig}) the condition  $\chi_h(\gamma)\in\Lambda_{k,0}$ gives 

\begin{eqnarray}
v_0+R(t_0) v - R(t)v_0 & \in & \ZZ \times \ZZ \label{n1} \\
z+\frac12 v_0^tJR(t_0)v-\frac12 v_0^tJR(t)v_0-\frac12(v^t R(-t_0)JR(t)v_0) & \in & \frac1{2k} \ZZ \label{n2}
\end{eqnarray}

Since $t\in 2\pi \ZZ$ then $R(t)\equiv Id$, thus  $R(t_0) v\in \ZZ \times \ZZ$ for $v\in \ZZ \times \ZZ$ 
which implies 
\begin{equation}
t_0 = \frac{\pi}2r  \qquad \mbox{ for  some } \quad  r\in \ZZ \label{t0}.
\end{equation}
 Now using this in (\ref{n2}) one gets 
 \begin{equation} \label{v0}
 v_0\in \frac1{2k}\ZZ \times \frac1{2k}\ZZ.
 \end{equation} 

Canonical computations show that 
$g=(\frac{\pi}2 r, \frac1{2k} p, \frac1{2k} q, s) \in \mathcal N_{G}(\Lambda_{k,0})$ 
for  all $r,p,q \in \ZZ$ and $s\in \RR$. 

\medskip

For $\Lambda_{k, \pi}$ an element $h=(t_0, v_0, z_0)\in G$ which belongs 
to $\mathcal N_{G} (\Lambda_{k, \pi})$ must satisfy equations (\ref{n1}) and (\ref{n2}). 
Observe that elements of the form $\gamma=(2\pi s, m, n, \frac1{2k} z )\in \Lambda_{k,\pi}$. 
Therefore $h$ must satisfy the conditions above (\ref{t0}, \ref{v0}). 

For $t=\pi s$ with $s\equiv 1$ (mod 2) the condition (\ref{n1}) implies 
that $v_0\in \frac12 \ZZ \times \frac12 \ZZ$. Finally usual computations give
 $\mathcal N_{G}(\Lambda_{k,\pi})=\frac{\pi}2\ZZ \ltimes (\frac12 \ZZ \times \frac12 \ZZ \times \RR)$.  

For the lattice $\Lambda_{k,\frac{\pi}2}$ notice that  we can use conditions obtained for the other two families of lattices. Thus assume that $h\in \mathcal N_{G}(\Lambda_{k, \frac{\pi}2})$ has the form $g=(\frac{\pi}2 r, \frac12 p, \frac12 q, z_0)$ for $r,p,q\in \ZZ$, $z_0\in \RR$. Thus we should analyse equations (\ref{n1}) and (\ref{n2}) for $t \in \pm \frac{\pi}2 + 2\pi \ZZ$. 

Condition (\ref{n1}) implies $p\equiv q$ (mod 2). Imposing this together with condition (\ref{n2}) accounts to $v_0 \in \frac12 (\ZZ \times \ZZ)$ for $k\geq 2$ or $v_0\in \ZZ \times \ZZ$ for $k=1$.
\end{proof}

Once one knows which isometries of $G$ are fiber preserving, to study the isometry group of
 $M_{k,i}$ one should  determine, among others, 
 which of these isometries act effectively on $M_{k,i}$ for $i=0,\pi,\pi/2$. 



Thus to determine the isometry group of the compact space $M_{k,i}$ we need to find the kernel of the map
$$\widetilde{\chi}: \mathcal{N}_{G}(\Lambda_{k,s}) \to \Iso(M_{k,i})\qquad \mbox{ such that }\quad h \to \widetilde{\chi}_{h}$$
and the kernel of the map 
$$\widetilde{\tau}: G \to \widetilde{\Le}(M_{k,s}) \qquad \mbox{ such that }\quad h \to \tau_h.$$

One obtains
\begin{equation}\label{isof}
Im(\widetilde{\chi}):=\widetilde{\Fe}(M_{k,i})\simeq  \mathcal{N}_{G}(\Lambda_{k,s}) /\{h\in \mathcal{N}_{G}(\Lambda_{k,s})\, :\, h=(2\pi s, 0, r): s\in \ZZ, r\in \RR\}\end{equation}

and 
\begin{equation}\label{lc}
Im(\widetilde{\tau}):= \widetilde{\Le}(M_{k,s})\simeq G/\{h\in G\,/\, h=(2\pi s, 0,z): s\in \ZZ, z\in \frac1{2k}\ZZ\}.
\end{equation}

\begin{thm}
Let $M_{k,s}$ denote the solvmanifolds of dimension  four as in (\ref{solvm}) equipped with the naturally reductive metric induced by the bi-invariant metric of $G$ given by $g$ (\ref{metric}). Then the isometry group of $M_{k,s}$ is given by
$$\Iso(M_{k,s})=\widetilde{\Fe}(M_{k,i})\cdot \widetilde{\Le}(M_{k,s})$$
where $\widetilde{\Fe}(M_{k,i})$ is the group in (\ref{isof}) and $\widetilde{\Le}(M_{k,s})$ is the group
 in (\ref{lc}).

Moreover 
\begin{itemize}
\item $\widetilde{\Le}(M_{k,s})$ is a normal subgroup  and 
\item $\widetilde{\mathcal N}(M_{k,s})\cap\widetilde{L}(M_{k,s})=
\{ \tau_Z\circ \widetilde{\chi}_\gamma, \, \mbox{ where } Z:=(0,0,0,z) \, z\in \RR,\, \gamma \in \Lambda_{k,s}\}$.
\end{itemize}
\end{thm}

\begin{rem} Notice that $\Iso_0(M_{k,i})$ has $G$  as universal covering.

Also note that $\RR \times \Heis_3(\RR)$ does not act by isometries on the quotients $M_{k,i}$ for any $k, i$.  
\end{rem}

Since the projection of the left-invariant vector field  $X_0- X_3$ to $M_{k,s}$ gives a timelike
 Killing vector field one gets the following fact.
\begin{cor}
 All of the compact spaces $M_{k,s}$ are stationary.
\end{cor}

\begin{rem}
 Theorem 4.1 in \cite{PZ} states that when the
identity component of the isometry group is non-compact and it has some timelike orbit,
then it must contain a non-trivial factor locally isomorphic to $SL(2, R)$ or to an oscillator
group.
\end{rem}

\subsection{Geodesics on $M_{k,s}$ }
Our aim here is to study the geodesics on the quotient spaces $M_{k,s}=G/\Lambda_{k,s}$ for 
$s=0,\pi, \pi/2$.  Since $M_{k,s}$ is  a naturally reductive space  the geodesics starting at $p(e)$ are
 precisely the projections of the geodesics of $G$ through the identity element $e$ 
(see Ch. 11 \cite{ON}). Any other geodesic of $G$ is the translation on the left of a geodesic through $e$, giving rise to any geodesic on the quotient. 

Let $\bar{\gamma}(t)=p\circ \gamma(t)$ denote a curve on $M_{k,s}$ with initial velocity
$$\bar{v}=\bar{\gamma}'(0)=dp_e (\gamma'(0)).$$ 
The tangent vector $\bar{v}$ is called 

\begin{itemize}
\item {\bf lightlike or null} if it has null norm.

\item {\bf spacelike} if it has positive norm.

\item {\bf timelike} if it has negative norm.
\end{itemize}

The curve $\bar{\gamma}$ is called {\em lightlike} (resp. {\em spacelike, timelike}) if
 its tangent vector is lightlike (resp. spacelike, timelike) at every point. 

Observe first that a tangent vector $X$ of $G$ of the
 form $X= \sum_{i=0}^3 a_i X_i$ for the left-invariant vector fields $X_i$, is null if it satisfies the condition:
\begin{equation}\label{nullv}
a_1^2+a_2^2+2a_0 a_3=0,
\end{equation}
while other tangent vectors on $G$ satisfying $a_1^2+a_2^2+2a_0 a_3>0$ or $<0$ are either spacelike or timelike respectively.

Let $\alpha$ denote a curve on $G$. Its projection will be denoted by $\bar{\alpha}=p\circ \alpha$. Observe that $\bar{\alpha}$ is self-intersecting if and only if there exist $t_0, t_1 \in \RR$ such that $\alpha(t_1)^{-1} \alpha(t_0)\in \Lambda_{k,s}$

\begin{lem}  Let $G$ denote a Lie group, let $K < G$ be a subgroup of $G$ let $p: G \to G/K$ denote the 
canonical projection. Let $\alpha: \RR \to G$ denote one-parameter subgroup of $G$.
If $p\circ \alpha$ is closed in $G/K$ then it is periodic.

\end{lem}
\begin{proof} Assume there exist $t_0, t_1\in \RR$ such that
$\bar{\alpha}(t_0)=\bar{\alpha}(t_1)$. Thus $\alpha(t_1)^{-1}\alpha(t_0)\in K$. Since 
$\alpha$ is a one-parameter subgroup it holds $\alpha(t_0-t_1)\in K$. Set $T=t_1-t_0$ then 
$\alpha(s+T)=\alpha(s) \alpha(T)$ and so  $\bar{\alpha}(t+T)=\bar{\alpha}(t)$ for all $t\in \RR$. 
  
\end{proof}

\begin{cor}
Let $G/K$ be a naturally reductive pseudo-Riemannian space. Then every closed geodesic in $G/K$ is periodic. 
\end{cor}

The next step is to apply this result to study periodic geodesics on the quotient spaces 
$M_{k,i}$ $i=0, \pi, \pi/2$. Geodesics on $M_{k,s}$ are induced by one-parameter subgroups 
of $G$ since the metric of $G$ is bi-invariant.

Indeed a geodesic $\alpha$ on $G$ through $e$ with tangent vector
 $X=\sum_{i=0}^3 a_i X_i$ gives rise
 to a closed geodesic on $M_{k,0}$ if and only if there exists $T\in \RR$ 
such that $\alpha(T)\in \Lambda_{k,0}$, which

\begin{itemize} 
 \item for  $a_0\neq 0$ gives the following condition
\begin{equation}\label{closednz}
\begin{array}{rcl}
a_0 T & \in & 2 \pi \ZZ \\
a_0^{-1} (R(a_0T) J - J)(a_1, a_2)^t & \in & \ZZ \times \ZZ \\
\displaystyle{(\frac{a_1^2+a_2^2}{2a_0} + a_3)T - \frac{a_1^2+a_2^2}{a_0^2} \sin(a_0T)} & \in &  \frac1{2k} \ZZ.
\end{array}
\end{equation}
Notice that if the first condition holds then $R(a_0 T)$ is the identity map so that 
$R(a_0T) J - J=0$ and the second condition is satisfied for all $a_1, a_2 \in \RR$. Since 
$a_0T\in 2 \pi \ZZ$ then $\sin(a_0T)=0$ and the third condition reduces to 
\begin{equation}\label{closed}
\displaystyle{\frac{||X||^2}{2a_0} T= (\frac{a_1^2+a_2^2}{2a_0} + a_3)}T \in  \frac1{2k} \ZZ.
\end{equation}
Hence if $a_0\neq 0$ the condition of $p \circ \alpha$ being closed on $M_{k,0}$
 reduces to (\ref{closed}). 

For spacelike or timelike geodesics, that is $||X||^2>0$ or $||X||^2<0$ respectively, where $||X||^2= < X, X >$ 
closed geodesics on $M_{k,0}$ are determined by the conditions
$$
a_0 T \in  2 \pi s \qquad \mbox{ and } \qquad \frac{||X||^2}{2a_0} T = 
\frac{m}{2k} \quad \mbox{ for }m,s \in \ZZ.
$$

\item  For $a_0=0$ notice the geodesic $\bar{\alpha}$ is closed if there exists $T\in \RR$ such that 
\begin{equation}\label{closedz}
\begin{array}{rcl}
(a_1 T, a_2 T)^t & \in & \ZZ \times \ZZ \\
a_3 T & \in &  \frac1{2k} \ZZ
\end{array}
\end{equation}
Thus on $G$ a null geodesic is $\alpha(s)=(0,0,0,a_3 s)$ which gives rise to a periodic geodesic on $M_{k,s}$ if and only if
$ a_3 T  \in   \frac1{2k} \ZZ$. 
Therefore
\begin{itemize}
\item  every lightlike geodesic on $M_{k,0}$ is closed.
\item there are periodic and non-closed timelike and spacelike geodesics on $M_{k,0}$.
\end{itemize}
\end{itemize}

\begin{thm} Let  $M_{k,i}$ denote  the solvmanifolds as in (\ref{solvm}). 
\begin{itemize}
\item Every null geodesic is periodic on $M_{k,i}$ for $i=0,\pi, \pi/2$.
\item There are  closed  and non closed timelike and spacelike geodesics on $M_{k, i}$ for $i=0,\pi, \pi/2$.
\end{itemize}
\end{thm}

 For the other families of lattices $\Lambda_{k,\pi}$ and $\Lambda_{k, \pi/2}$ one should 
modify the equations in (\ref{closednz}) and (\ref{closedz}) to get the condition for 
$\bar{\alpha}$ to  be closed. Analogous arguments prove all the assertions of the Theorem. One should notice that 
the analysis in these cases gives some extra geodesics once $a_0T= \pi m$ or $a_0T= \frac{\pi m}2$ for some $m\in \ZZ$.

\begin{rem} Every compact manifold $M_{k,i}$ is even-dimensional and orientable.  Compare with 
 Theorem 2 in  \cite{Ga}.

The Ricci tensor on $G$ verifies
$$Ric(X,X)= \frac12 a_0^2 \geq 0 \qquad \mbox{ for } \qquad X= a_0 \deat + V, \quad V \in span\{\deaz, \deax, \deay\} 
 $$
 and since $p$ is a local isometry, $G$ so as their quotients satisfy the lightlike and timelike 
convergence conditions. 
 
\end{rem}

\end{document}